\documentclass[11pt]{article}
\usepackage[margin=1.15in]{geometry}
\usepackage{amsmath,amssymb,amsthm,mathtools}
\usepackage{enumitem}
\usepackage[colorlinks=true,linkcolor=blue,citecolor=blue,urlcolor=blue]{hyperref}

\newtheorem{theorem}{Theorem}[section]
\newtheorem{lemma}[theorem]{Lemma}
\newtheorem{proposition}[theorem]{Proposition}

\newtheorem{conjecture}[theorem]{Conjecture}
\theoremstyle{definition}
\newtheorem{definition}[theorem]{Definition}
\newtheorem{assumption}[theorem]{Assumption}
\theoremstyle{remark}
\newtheorem{remark}[theorem]{Remark}
\newtheorem{question}{Question}

\newcommand{\Z}{\mathbb Z}
\newcommand{\R}{\mathbb R}
\newcommand{\Pb}{\mathbb P}
\newcommand{\E}{\mathbb E}
\newcommand{\cG}{\mathcal G}
\newcommand{\GG}{\mathbb G}
\newcommand{\PP}{\mathcal P}
\newcommand{\tp}{\widetilde p}
\newcommand{\one}{\mathbf 1}
\newcommand{\Ed}{\mathrm{Ed}}
\DeclareMathOperator{\homo}{hom}
\numberwithin{equation}{section}

\title{Trees in random sparse graphs\\ with a given degree sequence}
\author{Behzad Mehrdad}
\date{Revised version (v4) of arXiv:1401.0220}

\begin{document}
\maketitle

\begin{abstract}
Let $\GG^D$ be the set of simple graphs on $n$ vertices with degree sequence $D=(d_1,\dots,d_n)$, and, for $\frac12<a<1$, let $\GG^D_a$ be the union of $\GG^{\bar D}$ over all $\bar D$ with $|d_i-\bar d_i|\le d_i^a$ for every $i$. We compare a uniformly random graph from $\GG^D_a$, or from $\GG^D$, with the random graph with independent edges whose edge probabilities maximize the entropy under the constraint that the expected degrees are $D$. The comparison is made through the probabilities of embeddings of finite trees, weighted by a local normalization (the B-function) that depends on the degrees at the image of the tree. We prove convergence for $\GG^D_a$ under mild conditions on $D$, for $\GG^D$ in the dense regime and in the very sparse regime ($d_{\max}=o(|E|^{1/4})$), and for random bipartite graphs with given degrees under no conditions. For $\GG^D$ in general the result is reduced to a lower bound on the number of graphs, which is stated as a conjecture.
\end{abstract}

\tableofcontents

\section{Introduction}

Let $D=(d_1,\dots,d_n)$ be positive integers with $d_1\le\dots\le d_n$ and $M:=\sum_id_i$ even. Let $\GG^D$ be the set of simple graphs (no loops, no multiple edges) on the vertex set $[n]=\{1,\dots,n\}$ in which vertex $i$ has degree $d_i$; each has $M/2$ edges. Let $\cG_g(D)$ be uniform on $\GG^D$. What does $\cG_g(D)$ look like? We write $g=O(f)$, $g=o(f)$, $g=\Theta(f)$ in the usual sense as $n\to\infty$.

\emph{Dense graphs} have $|E|=\Theta(n^2)$. Barvinok and Hartigan \cite{BH13} showed that, under a mild condition ($\delta$-tameness), the number of graphs with degree sequence $D$ is governed by a maximum entropy problem and found its asymptotics. McKay and Wormald \cite{MW90} treated nearly regular dense degree sequences; see \cite{McK11} for subgraph probabilities and \cite{McK10} for a survey. Through graph limits \cite{LS06,BCLSV08,BCLSV12}, Chatterjee, Diaconis and Sly \cite{CDS11} proved that sequences of such graphs converge to graphons when the degree sequences converge to a degree function satisfying the Erd\H{o}s--Gallai condition for graph limits.

\emph{Very sparse graphs}, $d_{\max}^2=o(M)$, were studied by McKay \cite{McK85,McK10} via inclusion--exclusion and switchings, and by McKay and Wormald \cite{MW91} under weaker conditions. Gao, Su and Wormald \cite{GSW12} studied the probabilities of induced subgraphs in this regime; see also \cite{Gao11}.

\emph{Bipartite graphs with given degrees} are $0$-$1$ matrices with given row and column sums. Barvinok \cite{Bar10a,Bar10b,Bar12} and Barvinok--Hartigan \cite{BH10,BH12} studied them in the dense case and found the relation between their number and an entropy; Canfield--Greenhill--McKay \cite{CGM08} and Canfield--McKay \cite{CM10} gave formulas for nearly constant sums, and \cite{GMW06} treated sparse irregular sums.

\emph{Motivation.} For a finite graph $F$, the homomorphism density of Lov\'asz and Szegedy \cite{LS06} is $t(F,G)=\hom(F,G)/|V(G)|^{|V(F)|}$, and a sequence $(G_n)$ is Cauchy in the subgraph counting metric if $t(F,G_n)$ is Cauchy for every $F$. This normalizes all embeddings of $F$ by the same factor. For sparse graphs we believe that the normalization should be \emph{local} and depend on the embedding \cite{BR09}. We compare $\cG_g(D)$ with the independent model $\widetilde\cG(D)$ of \cite{CDS11}, and find the correct normalization for counting a tree $F$. Counting subgraphs with cycles is left open, since it is not completely understood even for models with independent edges (see \cite{JR11,Cha12}).

\emph{Plan.} First, we introduce the model $\cG_a(a,D)$ of graphs with an \emph{almost} given degree sequence and find the distribution of finite trees in it under mild conditions (Theorem~\ref{thm:almost}). Second, we return to $\cG_g(D)$: we reduce the problem to a lower bound on $|\GG^D|$ (Conjecture~\ref{conj:vol}) and prove the results unconditionally in the dense, the very sparse, and the bipartite cases. The method works from very sparse to dense graphs. In the bipartite case it needs no condition on the degrees, and we believe the same is true for general graphs.

\section{Main results}\label{sec:main}

Throughout, $c,c_k,C$ denote constants that do not depend on $n$ (and $c_k$ may depend on $k$), and $n$ is assumed large.

\subsection{Maximum entropy and the independent ensemble}\label{subsec:entropy}
Let $N=\binom n2$ and let
\[
\PP^D=\Big\{x=(x_{ij})_{1\le i<j\le n}\in[0,1]^N:\ \sum_{j\ne i}x_{ij}=d_i\ \ (i\in[n])\Big\},\qquad x_{ji}:=x_{ij},
\]
a polytope. Put $H(x)=-x\ln x-(1-x)\ln(1-x)$ and
\begin{equation}\label{eq:H1}
H_1(x)=\sum_{i<j}H(x_{ij}).
\end{equation}
\begin{proposition}\label{prop:EG}
The polytope $\PP^D$ contains a point with $0<x_{ij}<1$ for all $i<j$ if and only if $D$ satisfies the \emph{strict Erd\H{o}s--Gallai conditions}
\begin{equation}\label{eq:EG}
\sum_{i=n-k+1}^nd_i<k(k-1)+\sum_{i=1}^{n-k}\min\{k,d_i\},\qquad 1\le k\le n.
\end{equation}
\end{proposition}
\noindent Replacing ``$<$'' by ``$\le$'' gives the Erd\H{o}s--Gallai criterion \cite{EG60} for $D$ to be the degree sequence of a graph (see \cite{MP95}).

\begin{definition}\label{def:strict}
$D$ is a \emph{strict graphic sequence} if it satisfies \eqref{eq:EG}.
\end{definition}
If $D$ is strict graphic, $H_1$ is strictly concave and has a unique maximizer $\tp=(\tp_{ij})$ on $\PP^D$, which lies in the open cube (Section~\ref{sec:proofs}). Let $\widetilde\cG(D)$ be the random graph with independent edges, $\Pb(ij\in\widetilde\cG(D))=\tp_{ij}$. By definition of $\PP^D$,
\begin{equation}\label{eq:expdeg}
\sum_{j\ne i}\E\,\one_{\{ij\in\widetilde\cG(D)\}}=\sum_{j\ne i}\tp_{ij}=d_i,
\end{equation}
so $D$ is the vector of expected degrees of $\widetilde\cG(D)$. We show in Section~\ref{sec:proofs} that there are positive numbers $r_1\le\dots\le r_n$ with
\begin{equation}\label{eq:r}
\tp_{ij}=\frac{r_ir_j}{1+r_ir_j}.
\end{equation}

\subsection{Ordered trees and the B-function}\label{subsec:trees}
Let $\mathbb T_k$ be the set of trees on the vertex set $[k]$; by Cayley's theorem $|\mathbb T_k|=k^{k-2}$. Let $\mathbb S_k=\mathbb S_k^n$ be the set of injective maps $s:[k]\to[n]$. An \emph{ordered tree} is a pair $(s,T)\in\mathbb S_k\times\mathbb T_k$; we identify it with the subgraph $s(T)$ of $K_n$ together with its labelling. For $k=2$, $\mathbb S_2\times\mathbb T_2$ is the set of ordered pairs $(i,j)$, $i\ne j$.

\begin{definition}[B-function]\label{def:B}
For $u\in V(T)$ let $b_u$ be its degree in $T$. For a positive vector $\Delta=(\Delta_i)_{i\in[n]}$ put
\begin{equation}\label{eq:psi}
\psi(s,T,\Delta)=\prod_{u\in V(T)}\Delta_{s(u)}^{\,b_u-1}.
\end{equation}
We write $\psi(s,T,D)$ for $\Delta=D$, and $\psi(s,T,G)$ for $\Delta=$ the degree sequence of a graph $G$.
\end{definition}
\noindent \emph{Example.} Let $D=(3,1,2,1,3,3,4,3)$ and let $T$ be the path $1\!-\!3\!-\!2$ with middle vertex $3$, $s(1)=8$, $s(3)=7$, $s(2)=5$. Then $\psi(s,T,D)=d_8^{0}\,d_7^{1}\,d_5^{0}=4$.

\begin{remark}\label{rem:perm}
For a permutation $\pi$ of $[k]$, $\psi(s,T,D)=\psi(s\circ\pi^{-1},\pi(T),D)$. Hence each tree with image $s(T)$ occurs $k!$ times in $\mathbb S_k\times\mathbb T_k$ with the same B-function.
\end{remark}
\noindent For $k=3$ the ordered tree is a path of two edges and $\psi=d_j$, $j$ the middle vertex. If all $\tp_{ij}$ are of the same order $p$, then $M\psi(s,T,D)\asymp n^2p\,(np)^{k-2}\asymp n^k\tp(s,T)$, so $L_X\asymp n^{-k}\sum_{(s,T)}|p_X(s,T)/\tp(s,T)-1|$ is a relative error averaged over embeddings. In general the B-function is the normalization that makes $\frac1M\sum_{(s,T)}\tp(s,T)/\psi(s,T,D)$ equal to $k^{k-2}$ up to a small error (Proposition~\ref{prop:norm}). If every $d_i=d$, then $\tp_{ij}=d/(n-1)$ and $\widetilde\cG(D)$ is the Erd\H{o}s--R\'enyi graph $G(n,d/(n-1))$.

For a random graph $X$ on $[n]$ write
\begin{equation}\label{eq:indic}
\one_s(T,G)=\prod_{\{u,v\}\in E(T)}\one_{\{s(u)s(v)\in E(G)\}}(G),\qquad p_X(s,T)=\E\,\one_s(T,X).
\end{equation}
We write $p_a,p_g,p_q,p_b$ and $\tp$ for $X=\cG_a(a,D)$, $\cG_g(D)$, $\cG_q(D)$, $\cG_b(D_1,D_2)$ and $\widetilde\cG(D)$ respectively (these models are defined below). For a random graph $X$ and a set $\mathcal A\subseteq\mathbb S_k\times\mathbb T_k$ put
\begin{equation}\label{eq:LX}
L_X(k,D)=\frac1M\sum_{(s,T)}\frac{|p_X(s,T)-\tp(s,T)|}{\psi(s,T,D)}.
\end{equation}
We write $L_{\mathbf a}$, $L_{\mathbf g}$, $L_{\mathbf q}$, $L_{\mathbf b}$ for $X=\cG_a$, $\cG_g$, $\cG_q$, $\cG_b$.

\subsection{Graphs with an almost given degree sequence}\label{subsec:almost}
Let $S_k=\sum_{i=n-d_k}^nd_i$ be the sum of the $d_k+1$ largest degrees, and
\begin{equation}\label{eq:ell}
\ell(D)=\max\{k\in[n]:\ S_k<M/2\}.
\end{equation}
\begin{assumption}\label{ass:type}
For some $\varepsilon>0$ and $\nu>0$:
\begin{enumerate}[label=(\arabic*),leftmargin=2em]
\item $D$ is strict graphic;
\item $M$ is even and $n^{1+\varepsilon}\le M$;
\item the set in \eqref{eq:ell} is non-empty and
\begin{equation}\label{eq:nu}
\sqrt{\frac nM}\,\big(d_n-d_{\ell(D)}+1\big)<n^{-\nu}.
\end{equation}
\end{enumerate}
\end{assumption}
\begin{definition}
$D$ is \emph{strict graphic of type $(\varepsilon,\nu)$} if it satisfies Assumption~\ref{ass:type}, and \emph{of type $\varepsilon$} if it satisfies (1) and (2).
\end{definition}
\begin{remark}\label{rem:examples}
If $d_n(d_n+1)<M/2$ (in particular if $d_n^2=o(M)$), then $S_n\le(d_n+1)d_n<M/2$, so $\ell(D)=n$, $d_n-d_\ell=0$ and \eqref{eq:nu} reads $\sqrt{n/M}<n^{-\nu}$, which holds for every $\nu<\varepsilon/2$. Similarly, if $2n^{\alpha}<d_i<n^{\beta}$ for all $i$ and $\beta<(1+\alpha)/2$, then $M>2n^{1+\alpha}>2(n^\beta+1)n^\beta\ge 2S_n$ for large $n$, so again $\ell(D)=n$ and $D$ has type $(\varepsilon,\nu)$ for every $\nu<\varepsilon/2$ with $\varepsilon=\alpha$.
\end{remark}

Fix $a\in(\frac12,1)$ and put $a_1=a-\frac12$. Let
\begin{equation}\label{eq:Ga}
\GG^D_a=\bigcup_{\bar D}\GG^{\bar D},\qquad\text{over all }\bar D\text{ with }|d_i-\bar d_i|\le d_i^a\ \ (1\le i\le n),
\end{equation}
and let $\cG_a(a,D)$ be uniform on $\GG^D_a$. (Here $\bar D$ ranges over graphic sequences; $\bar d_i\le 2d_i$ always, and $\bar d_i=0$ is possible only if $d_i=1$. Earlier versions used the strict inequality $<$, which for $d_i=1$ forces $\bar d_i=1$; the proof of Proposition~\ref{prop:lower} below works with $\le$.)

\begin{theorem}\label{thm:almost}
Let $D$ be strict graphic of type $(\varepsilon,\nu)$ and let $\frac12<a<\frac\nu4+\frac12$. Then for every $k\ge2$
\begin{equation}\label{eq:almost}
L_{\mathbf a}(a,k,D)\le c_k\,n^{\,a-\frac{\nu+1}2}=o(1).
\end{equation}
\end{theorem}

\begin{remark}\label{rem:sym}
$\psi(s,T,D)$ is invariant under relabelling (Remark~\ref{rem:perm}), and $\mathbb S_k\times\mathbb T_k$ covers $\mathcal T(k,n)=\bigcup_{T\in\mathbb T_k}\homo(T,K_n)$ exactly $k!$ times. Summing over $\mathcal T(k,n)$ instead gives $L_{\mathbf a}/k!$.
\end{remark}
\begin{remark}\label{rem:forest}
The theorem holds, with the same proof, for forests; the B-function of a forest is the product of the B-functions of its components and the normalization is $M^{-m}$ for $m$ components. For $k=2$ and two components, $\mathbb S_2\times\mathbb T_2$ is the set of ordered pairs $(i,j)$, the B-function of each is $1$, and the statement is
\[
\frac1{M^2}\sum\big|p_{a,(ij)(kl)}-\tp_{ij}\tp_{kl}\big|=o(1),
\]
the sum over ordered pairs of disjoint edges $(i,j),(k,l)$ and $p_{a,(ij)(kl)}=\Pb(ij,kl\in\cG_a)$.
\end{remark}
\begin{remark}\label{rem:k23}
For $k=2$ the B-function is $1$ and the theorem says $\frac1M\sum_{i\ne j}|p_{a,(ij)}-\tp_{ij}|\le c\,n^{a-\frac{\nu+1}2}$, where $p_{a,(ij)}=\Pb(ij\in\cG_a)$: on average the edge probabilities of $\cG_a$ are those of the maximum entropy model. For $k=3$ the ordered trees are paths $i\!-\!j\!-\!k$ with $\psi=d_j$, and $L_{\mathbf a}(a,3,D)/3!=\sum_j\sum_{\{i,k\}}\frac1{Md_j}\big|p_{a,(ij)(jk)}-\tp_{ij}\tp_{jk}\big|$, with $p_{a,(ij)(jk)}=\Pb(ij,jk\in\cG_a)$.
\end{remark}
\begin{remark}\label{rem:best}
If $|\log r_i|\le c\log n$ for all $i$ (this holds if $d_n(d_n+1)<M/2$, see Remark~\ref{rem:r}) and $a_1\le\frac14$, the proof gives the better bound $L_{\mathbf a}\le c_k(n/M)^{1/4}n^{a_1}$. We believe this is the correct order, and that the last condition in Assumption~\ref{ass:type} is unnecessary (Conjecture~\ref{conj:r}).
\end{remark}

\subsection{Graphs with a given degree sequence}\label{subsec:exact}
Let $\cG_g(D)$ be uniform on $\GG^D$. Let $D$ be strict graphic and put
\[
\pi(D):=\Pb\big(\widetilde\cG(D)\in\GG^D\big).
\]
We show in Section~\ref{sec:proofs} that $\Pb(\widetilde\cG(D)=G)=e^{-H_1(\tp)}$ for every $G\in\GG^D$, hence $\pi(D)=|\GG^D|e^{-H_1(\tp)}\le1$.
\begin{conjecture}\label{conj:vol}
Let $D$ be strict graphic of type $\varepsilon$. There is $\eta>0$ independent of $n$ such that
\begin{equation}\label{eq:vol}
e^{-\eta n\log n}e^{H_1(\tp)}\le|\GG^D|\le e^{H_1(\tp)},\qquad\text{i.e.}\qquad e^{-\eta n\log n}\le\pi(D)\le1.
\end{equation}
\end{conjecture}
\noindent The upper bound is the easy part; the lower bound is open in general.
\begin{conjecture}\label{conj:trees}
Let $D$ be strict graphic of type $\varepsilon$. Then for every $k\ge2$
\[
L_{\mathbf g}(k,D)\le c_k\sqrt{\frac{n\log n}M}.
\]
\end{conjecture}
\begin{proposition}\label{prop:conj}
Let $D$ be strict graphic with $M\ge n$ and suppose $\pi(D)\ge e^{-K_0}$. Then $L_{\mathbf g}(k,D)\le c_k\big(\sqrt{(K_0+\log M)/M}+n/M\big)$. In particular Conjecture~\ref{conj:vol} implies Conjecture~\ref{conj:trees}.
\end{proposition}

\subsection{Dense graphs with a given degree sequence}\label{subsec:dense}
\begin{definition}\label{def:dense}
Let $c_3>0$ and $c_1,c_2\in(0,1)$. $D$ satisfies the \emph{dense Erd\H{o}s--Gallai conditions of type $(c_1,c_2,c_3)$} if
\begin{enumerate}[label=(\arabic*),leftmargin=2em]
\item $c_2(n-1)\le d_i\le c_1(n-1)$ for every $i$, and
\item $\displaystyle\frac1{n^2}\inf_{\substack{B\subseteq[n]\\|B|\ge c_2n}}\Big\{\sum_{j\notin B}\min\{d_j,|B|\}+|B|(|B|-1)-\sum_{i\in B}d_i\Big\}\ge c_3$.
\end{enumerate}
\end{definition}
\noindent By \cite{BH13,CDS11} this makes $D$ \emph{$\delta$-tame} for some $\delta=\delta(c_1,c_2,c_3)>0$, i.e.\ $\delta\le\tp_{ij}\le1-\delta$.
\begin{theorem}\label{thm:dense}
If $D$ satisfies the dense Erd\H{o}s--Gallai conditions then Conjectures~\ref{conj:vol} and~\ref{conj:trees} hold, so that
\begin{equation}\label{eq:densecase}
L_{\mathbf g}(k,D)\le c_k\Big(\frac{\log n}n\Big)^{1/2}.
\end{equation}
Moreover, for $\frac12<a<\frac12+\frac1{12}$,
\begin{equation}\label{eq:densea}
L_{\mathbf a}(a,k,D)\le c_k\,n^{-\frac34+a}.
\end{equation}
\end{theorem}
The first part uses the enumeration of \cite{BH13} (see Section~\ref{subsec:pf-dense}).

\subsection{Very sparse graphs with a given degree sequence}\label{subsec:sparse}
\begin{definition}\label{def:vs}
$D$ is \emph{very sparse} if $d_{\max}^2=o(M)$.
\end{definition}
Put $q_{ij}=\dfrac{d_id_j}{M+d_id_j}$ ($i\ne j$) and let $\cG_q(D)$ have independent edges with probabilities $q_{ij}$. This is an alternative to $\widetilde\cG(D)$ that does not require a maximum entropy computation: it is of the form \eqref{eq:r} with $r_i=d_i/\sqrt M$. Let $p_q(s,T)=\E\one_s(T,\cG_q(D))$ and let $L_{\mathbf q}$ be \eqref{eq:LX} with $p_q$.
\begin{theorem}\label{thm:sparse}
Let $D$ be strict graphic of type $\varepsilon$ and let $a_1=a-\frac12$ satisfy $(n/M)^{1/2}n^{3a_1}\ll1$. If $d_n^2=o(M)$ then part (1) holds; if $d_n=o(M^{1/4})$ then (1)--(3) hold:
\begin{enumerate}[label=(\arabic*),leftmargin=2em]
\item $L_{\mathbf a}(a,k,D)\le c_k\,(n/M)^{1/4}n^{a_1}$;
\item $L_{\mathbf g}(k,D)\le c_k\Big((n/M)^{1/4}n^{a_1}+\dfrac{d_n^2}M\Big)$;
\item $L_{\mathbf q}(k,D)\le c_k\Big((n/M)^{1/4}n^{a_1}+\dfrac{d_n^2}M\Big)$.
\end{enumerate}
\end{theorem}
\begin{remark}
Gao, Su and Wormald \cite{GSW12} compare $\cG_g(D)$ with $\cG_q(D)$ by enumeration. Here $\cG_g(D)$ is compared with the maximum entropy model. Part (3) with $k=2$ shows $\tp_{ij}\simeq q_{ij}$ on average. Note that $\sum_jq_{ij}\ne d_i$ in general, so $(q_{ij})\notin\PP^D$; but $d_n^2=o(M)$ makes it close to $\PP^D$ (Lemma~\ref{lem:q}). The proof of (2) and (3) uses McKay's enumeration formula, which is why $d_n=o(M^{1/4})$ is assumed there.
\end{remark}

\subsection{Bipartite graphs with an exact degree sequence}\label{subsec:bip}
Let the vertex set be $V_1\sqcup V_2$ with $|V_i|=n_i$, $n=n_1+n_2$, and let $D_i=(d_{i,1}\le\dots\le d_{i,n_i})$ be positive integers with $\sum_jd_{1,j}=\sum_jd_{2,j}$. Let $\GG^{D_1,D_2}$ be the set of bipartite graphs with these degrees (adjacency matrices with $0$-$1$ entries and row sums $D_1$, column sums $D_2$), assumed non-empty, put $M=\sum_jd_{1,j}+\sum_jd_{2,j}$ (twice the number of edges) and let $D=(D_1,D_2)$ regarded as a degree vector on $n$ vertices. Let $\cG_b(D_1,D_2)$ be uniform on $\GG^{D_1,D_2}$.

Let $\PP^{D_1,D_2}$ be the polytope of matrices $x\in[0,1]^{n_1\times n_2}$ with row sums $d_{1,i}$ and column sums $d_{2,j}$, and $H_2(x)=\sum_{i,j}H(x_{ij})$, a strictly concave function. If $\PP^{D_1,D_2}$ contains a point in the open cube $(0,1)^{n_1\times n_2}$, then $H_2$ has a unique maximizer $\tp$, of the form $\tp_{ij}=r_{1,i}r_{2,j}/(1+r_{1,i}r_{2,j})$ (the proof is the same as that of Proposition~\ref{prop:EG}). Let $\widetilde\cG(D_1,D_2)$ have independent edges $\Pb(ij\in\widetilde\cG)=\tp_{ij}$, $i\in V_1$, $j\in V_2$; there are no edges inside $V_1$ or $V_2$. Let $\mathcal T_b(k,n)$ be the set of ordered trees $(s,T)\in\mathbb S_k\times\mathbb T_k$ such that $s(T)$ has no edge with both ends in $V_1$ or both in $V_2$, and let $p_b(s,T)=\E\one_s(T,\cG_b)$, $\tp(s,T)=\E\one_s(T,\widetilde\cG(D_1,D_2))$,
\[
L_{\mathbf b}(k,D_1,D_2)=\frac1M\sum_{(s,T)\in\mathcal T_b(k,n)}\frac{|p_b(s,T)-\tp(s,T)|}{\psi(s,T,D)}.
\]
\begin{theorem}\label{thm:bip}
If $\PP^{D_1,D_2}$ contains a point in the open cube, then
\[
L_{\mathbf b}(k,D_1,D_2)\le c_k\Big(\frac{n\log n}M\Big)^{1/2}.
\]
\end{theorem}
\begin{remark}
In the bipartite case the lower bound of Conjecture~\ref{conj:vol} is available from \cite{Bar10a}, so the theorem holds with no condition on the degrees. A counterpart of Theorem~\ref{thm:almost} for bipartite graphs holds as well without conditions; we do not repeat it. For bipartite graphs the strict Erd\H{o}s--Gallai conditions reduce to a simpler system (Gale--Ryser type); we simply assume that the polytope has an interior point.
\end{remark}

\subsection{Open questions}\label{subsec:questions}
\begin{question}
Can the results be proved under fewer assumptions? In particular, can the last two parts of Assumption~\ref{ass:type} be dropped, so that only the strict Erd\H{o}s--Gallai conditions are used (see Conjecture~\ref{conj:r})?
\end{question}
\begin{question}
What can be said about triangle counts in $\cG_g(D)$? About other subgraphs?
\end{question}
\begin{question}
What is the correct way of counting subgraphs, in the sense that one can define a metric or topology on sparse graphs using weighted subgraph counts?
\end{question}
\begin{question}
If so, are there limiting objects for this metric?
\end{question}
\begin{question}
Can the maximum entropy model be used to define a limiting object for graphs with a given degree sequence?
\end{question}

\section{Proofs}\label{sec:proofs}

Throughout this section $n$ is large, $c,C$ denote absolute constants and $c_k$ a constant depending only on $k$ (and on the fixed parameter $a$ where it appears); they may change from line to line. We write $d_{\max}=d_n$; all graphs are simple.

\subsection{The maximum entropy point}\label{subsec:pf-entropy}

\begin{proof}[Proof of Proposition~\ref{prop:EG}]
\emph{Necessity.} Suppose $x\in\PP^D$ has $0<x_{ij}<1$ for all $i<j$. Fix $k$ and let $B$ be a set of $k$ vertices. Then
\[
\sum_{i\in B}d_i=\sum_{i\in B}\sum_{l\in B\setminus\{i\}}x_{il}+\sum_{l\notin B}\sum_{i\in B}x_{il}.
\]
The first double sum has $k(k-1)$ terms, each $<1$, so it is $<k(k-1)$. For $l\notin B$, $\sum_{i\in B}x_{il}<k$ and $\sum_{i\in B}x_{il}\le d_l$, so it is $\le\min\{k,d_l\}$. Taking $B$ to be the set of $k$ vertices of largest degree gives \eqref{eq:EG}. (Also $d_i>0$, so $d_i\ge1$.)

\emph{Sufficiency.} This is \cite[Lemma 12.2]{BH13}: if \eqref{eq:EG} holds then $\PP^D$ has non-empty interior (relative to its affine hull), i.e.\ contains a point of $(0,1)^N$.
\end{proof}

\begin{proof}[Existence of $\tp$ and of the $r_i$]
Let $D$ be strict graphic. $H_1$ is strictly concave and continuous on the compact convex set $\PP^D$, so it has a unique maximizer $\tp$. If some coordinate of $\tp$ were $0$ or $1$, then, moving from $\tp$ towards an interior point $y$ of $\PP^D$, the derivative of $H$ at $0$ and at $1$ is $+\infty$ in the inward direction and the other coordinates have bounded derivative, so $H_1$ would increase: a contradiction. Hence $0<\tp_{ij}<1$. The gradient of $H_1$ at the interior maximizer is orthogonal to the tangent space of the affine hull of $\PP^D$, which is the orthogonal complement of the span of the vectors $(\one_{\{i\in e\}})_{e}$, $i\in[n]$. Hence there are $\lambda_i\in\R$ with
\[
\partial_{ij}H_1(\tp)=\log\frac{1-\tp_{ij}}{\tp_{ij}}=\lambda_i+\lambda_j .
\]
With $r_i=e^{-\lambda_i}$ this is \eqref{eq:r}. Relabel the vertices so that $d_1\le\dots\le d_n$.
\end{proof}

\begin{lemma}\label{lem:prob}
For every graph $G$ on $[n]$ with degrees $d_i(G)$,
\begin{equation}\label{eq:probG}
\Pb(\widetilde\cG(D)=G)=\frac{\prod_ir_i^{d_i(G)}}{\prod_{i<j}(1+r_ir_j)} .
\end{equation}
In particular $\Pb(\widetilde\cG(D)=G)$ depends only on the degree sequence of $G$, and for every $G\in\GG^D$
\begin{equation}\label{eq:probH}
\Pb(\widetilde\cG(D)=G)=e^{-H_1(\tp)},\qquad\text{so}\qquad \pi(D)=|\GG^D|\,e^{-H_1(\tp)} .
\end{equation}
Consequently, conditionally on $\{\widetilde\cG(D)\in\GG^D\}$, $\widetilde\cG(D)$ is uniform on $\GG^D$.
\end{lemma}
\begin{proof}
By \eqref{eq:r}, $\tp_{ij}=r_ir_j/(1+r_ir_j)$ and $1-\tp_{ij}=1/(1+r_ir_j)$, so
$\Pb(\widetilde\cG=G)=\prod_{ij\in E(G)}r_ir_j\prod_{i<j}(1+r_ir_j)^{-1}$, and $\prod_{ij\in E(G)}r_ir_j=\prod_ir_i^{d_i(G)}$. This is \eqref{eq:probG}. For $G\in\GG^D$ take logarithms and use $\sum_{i<j}\tp_{ij}(\log r_i+\log r_j)=\sum_id_i\log r_i$ (from \eqref{eq:expdeg}):
\[
\begin{aligned}
-H_1(\tp)&=\sum_{i<j}\big[\tp_{ij}\log\tp_{ij}+(1-\tp_{ij})\log(1-\tp_{ij})\big]\\
&=\sum_{i<j}\tp_{ij}\log\frac{\tp_{ij}}{1-\tp_{ij}}+\sum_{i<j}\log(1-\tp_{ij})
=\sum_id_i\log r_i-\sum_{i<j}\log(1+r_ir_j),
\end{aligned}
\]
which is the logarithm of the right side of \eqref{eq:probG} for $G\in\GG^D$.
\end{proof}

\begin{lemma}\label{lem:r}
Let $D$ be strict graphic, $d_1\le\dots\le d_n$, $n\ge3$. Then:
\begin{enumerate}[label=(\alph*),leftmargin=2em]
\item $r_1\le\dots\le r_n$, and $r_i=r_j$ if and only if $d_i=d_j$;
\item $r_1r_n>\frac1n$;
\item (no large gaps) if $r_{k+1}>n$ then $r_{k+1}<n^4r_k$;
\item if $r_k>n^2$ then $\sum_{i=n-d_k}^nd_i\ge \frac M2$.
\end{enumerate}
\end{lemma}
\begin{proof}
(a) For $i\ne j$,
\[
d_j-d_i=\sum_{m\ne i,j}\Big(\frac{r_jr_m}{1+r_jr_m}-\frac{r_ir_m}{1+r_ir_m}\Big)=(r_j-r_i)\sum_{m\ne i,j}\frac{r_m}{(1+r_ir_m)(1+r_jr_m)} ,
\]
and the sum is positive, so $d_j-d_i$ and $r_j-r_i$ have the same sign (and vanish together).

(b) By (a), $\tp_{ij}$ is increasing in $i$ and in $j$, so $1\le d_1=\sum_{j\ge2}\tp_{1j}\le(n-1)\tp_{1n}<n\,r_1r_n$.

(c) Suppose $r_{k+1}>n$ and $r_{k+1}\ge n^4r_k$. Put $A=\{k+1,\dots,n\}$, $J=\{j\le k\}$, and split $J=B\cup J_1$ with $B=\{j\le k:\ r_jr_k\le n^{-2}\}$ (an initial segment) and $J_1=J\setminus B$. We claim
\begin{itemize}[leftmargin=2em]
\item for $i\ne l$ in $A$: $r_ir_l\ge r_{k+1}^2>n^2$, hence $1-\tp_{il}<n^{-2}$;
\item for $i\in A$, $l\in J_1$: $r_ir_l\ge n^4r_kr_l>n^2$, hence $1-\tp_{il}<n^{-2}$;
\item for $j\in B$ and $l\in B\cup J_1$, $l\ne j$: $r_jr_l\le r_jr_k\le n^{-2}$, hence $\tp_{jl}<n^{-2}$.
\end{itemize}
Consider the integer $U=\sum_{i\in A}d_i-\sum_{j\in B}d_j$. The terms $\tp_{ij}$ with $i\in A,j\in B$ occur in both sums and cancel, so
\[
U=\sum_{i\in A}\Big(\sum_{l\in A\setminus i}\tp_{il}+\sum_{l\in J_1}\tp_{il}\Big)-\sum_{j\in B}\Big(\sum_{l\in B\setminus j}\tp_{jl}+\sum_{l\in J_1}\tp_{jl}\Big)
=N_0-E ,
\]
where $N_0=|A|(|A|-1)+|A||J_1|$ is an integer and $E=\sum(1-\tp_{il})+\sum\tp_{jl}$ is a sum of $|A|(|A|-1+|J_1|)+|B|(|B|-1+|J_1|)\le(|A|+|B|)(n-1)<n^2$ terms, each in $(0,n^{-2})$. Hence $0<E<1$ and $U=N_0-E\notin\Z$, a contradiction. (The number of terms is positive because $|A|\ge1$ and $n\ge3$ give at least one term.)

(d) Let $I=\{i:\ r_i<n^{-1}\}$, an initial segment by (a). For $i\notin I$, $i\ne k$: $r_ir_k>n$, so $\tp_{ik}>\frac n{n+1}$. Hence $d_k>(n-|I|-1)\frac n{n+1}>n-|I|-2$ and, $d_k$ being an integer, $|I|\ge n-d_k-1$. For $i\ne j$ in $I$, $\tp_{ij}<n^{-2}$, hence $\sum_{i\in I}\sum_{j\in I\setminus i}\tp_{ij}<1$ and
\[
\sum_{i\in I}d_i<1+\sum_{i\in I}\sum_{j\notin I}\tp_{ij}\le1+\sum_{j\notin I}d_j=1+M-\sum_{i\in I}d_i .
\]
So $2\sum_{i\in I}d_i<M+1$; $M$ is even, so $\sum_{i\in I}d_i\le M/2$. Since $I\supseteq\{1,\dots,n-d_k-1\}$, $\sum_{i\ge n-d_k}d_i=M-\sum_{i<n-d_k}d_i\ge M-\sum_{i\in I}d_i\ge M/2$.
\end{proof}

\begin{lemma}\label{lem:logr}
Let $D$ be strict graphic of type $(\varepsilon,\nu)$ and $\ell=\ell(D)$. Then
\begin{equation}\label{eq:logr}
|\log r_i|\le 4\,(d_n-d_\ell+1)\log n\qquad(i\in[n]),
\end{equation}
and, with $\delta:=\min_{i\ne j}\min\{\tp_{ij},1-\tp_{ij}\}$,
\begin{equation}\label{eq:logdelta}
|\log\delta|\le 8\,(d_n-d_\ell+1)\log n .
\end{equation}
\end{lemma}
\begin{proof}
By definition of $\ell$, $\sum_{i\ge n-d_\ell}d_i<M/2$, so $r_\ell\le n^2$ by Lemma~\ref{lem:r}(d). By (a), $r_i=r_j$ iff $d_i=d_j$, so the numbers $r_\ell\le\dots\le r_n$ take $q\le d_n-d_\ell+1$ distinct values $v_1<\dots<v_q$, with $v_1\le n^2$. If $v_{t+1}>n$, then $v_{t+1}<n^4v_t$ by (c) (apply it to the last index with value $v_t$); otherwise $v_{t+1}\le n$. Hence $v_{t+1}\le\max\{n,n^4v_t\}$ and by induction $v_t\le n^{4t-2}$, so
\[
\log r_n\le(4q-2)\log n\le\big(4(d_n-d_\ell)+2\big)\log n .
\]
By (b), $-\log r_1<\log n+\log r_n$. Since $r_1\le r_i\le r_n$, $|\log r_i|\le\log n+(\log r_n)^+\le(4(d_n-d_\ell)+3)\log n$, which is \eqref{eq:logr}. Finally $\min\{\tp_{ij},1-\tp_{ij}\}\ge\frac12\min\{r_ir_j,1\}$, so $|\log\delta|\le\log2+2\max_i|\log r_i|$, and \eqref{eq:logdelta} follows for large $n$.
\end{proof}

\begin{remark}\label{rem:r}
If $d_n(d_n+1)<M/2$ then $\ell=n$, and Lemma~\ref{lem:logr} gives $|\log r_i|\le 4\log n$ (indeed $r_n\le n^2$ and $r_1>n^{-3}$, so $|\log r_i|\le3\log n$).
\end{remark}
\begin{conjecture}\label{conj:r}
For every strict graphic sequence $D$ with $M\ge n^{1+\varepsilon}$ one has $|\log r_i|\le c\log n$ for all $i$.
\end{conjecture}
\noindent If Conjecture~\ref{conj:r} holds then the last condition of Assumption~\ref{ass:type} can be dropped in the proofs below (see Remark~\ref{rem:best}).

\subsection{Weighted tree sums}\label{subsec:pf-trees}

For $\Delta=(\Delta_i)$, a weight function $w=(w_{ij})_{i\ne j}$ with $w_{ij}=w_{ji}\in[0,1]$, and an ordered tree $(s,T)$, put $w(s,T)=\prod_{\{u,v\}\in E(T)}w_{s(u)s(v)}$. For a graph $G$ take $w_{ij}=\one_{\{ij\in E(G)\}}$, so that $w(s,T)=\one_s(T,G)$. For $w=\tp$ we have $w(s,T)=\E\one_s(T,\widetilde\cG)=\tp(s,T)$. We use the convention $\psi=\Delta_{s(u)}^{-1}$ for a one-vertex tree $T=\{u\}$ ($b_u=0$), so that formula \eqref{eq:psi} is valid for $k=1$.

\begin{lemma}[leaf removal]\label{lem:peel}
Let $\Delta_i\ge1$, $M_\Delta=\sum_i\Delta_i$, $0\le\tau<1$, and let $w$ be symmetric with values in $[0,1]$ and
\begin{equation}\label{eq:wdeg}
(1-\tau)\Delta_i\le\sum_{x\ne i}w_{ix}\le\Delta_i\qquad(i\in[n]).
\end{equation}
For $T\in\mathbb T_k$ put $F_w(T)=\sum_{s\in\mathbb S_k}w(s,T)/\psi(s,T,\Delta)$. Then
\begin{equation}\label{eq:Fbounds}
(1-\tau)^{k-1}M_\Delta-\tfrac12nk^2\ \le\ F_w(T)\ \le\ M_\Delta .
\end{equation}
Moreover, if $T'$ is a tree with $m\ge1$ vertices and $u\in V(T')$, then $\sum_{s'}w(s',T')\,\Delta_{s'(u)}^{-1}/\psi(s',T',\Delta)\le n$, and for any $\theta\in\R$, $\sum_{s'}w(s',T')\,\Delta_{s'(u)}^{\theta}/\psi(s',T',\Delta)\le\sum_x\Delta_x^{1+\theta}$.
\end{lemma}
\begin{proof}
Let $k\ge2$, $v$ a leaf of $T$ attached to $u$, $T'=T-v$. Since $b_v=1$, and the degree of $u$ drops by one, $\psi(s,T,\Delta)=\psi(s',T',\Delta)\,\Delta_{s'(u)}$ where $s'=s|_{V(T')}$. Summing over $s(v)=x\notin s'(V(T'))$,
\[
F_w(T)=\sum_{s'}\frac{w(s',T')}{\psi(s',T',\Delta)}\cdot\frac{\mathrm{Ext}(s')}{\Delta_{s'(u)}},\qquad \mathrm{Ext}(s')=\sum_{x\notin s'(V(T'))}w_{s'(u)x}.
\]
By \eqref{eq:wdeg}, $\mathrm{Ext}(s')\le\Delta_{s'(u)}$, and, since $s'(V(T'))\setminus\{s'(u)\}$ has $k-2$ elements and $w\le1$, $\mathrm{Ext}(s')\ge(1-\tau)\Delta_{s'(u)}-(k-2)$. Therefore
\begin{equation}\label{eq:step}
(1-\tau)F_w(T')-(k-2)\,R_u(T')\ \le\ F_w(T)\ \le\ F_w(T'),\qquad R_u(T'):=\sum_{s'}\frac{w(s',T')}{\psi(s',T',\Delta)\Delta_{s'(u)}} .
\end{equation}
For the sums with a weight $\Delta_{s'(u)}^\theta$ attached to a vertex $u$ of $T'$: if $T'$ has at least two vertices it has at least two leaves, so we can remove a leaf $v\ne u$; the computation above (upper bound only) shows that the sum does not increase. After removing all vertices except $u$ we are left with $\sum_x\Delta_x^{\theta}\psi(x)^{-1}=\sum_x\Delta_x^{1+\theta}$ by our convention for one-vertex trees. For $\theta=-1$ this is $n$, so $R_u(T')\le n$. Finally $F_w(T_2)=\sum_{i\ne j}w_{ij}\in[(1-\tau)M_\Delta,M_\Delta]$ and \eqref{eq:step} iterated gives \eqref{eq:Fbounds}.
\end{proof}

\begin{lemma}[wedge sums]\label{lem:wedge}
Let $\alpha=(s_1,T_1)$ and $\beta=(s_2,T_2)$ be ordered trees of size $k$ whose images share at least one edge, and let $H=s_1(T_1)\cup s_2(T_2)$, $V=V(H)$, $j=|V|$. Then $k\le j\le2k-2$, and there is a tree $H'$ with vertex set $V$ and $s_1(T_1)\subseteq H'\subseteq H$ such that, for every $\Delta$ with all $\Delta_i\ge1$,
\[
\psi(H',\Delta)\le\psi(\alpha,\Delta)\,\psi(\beta,\Delta),\qquad\psi(H',\Delta):=\prod_{x\in V}\Delta_x^{\deg_{H'}(x)-1}.
\]
\end{lemma}
\begin{proof}
$H$ is connected and $|V(s_1(T_1))\cap V(s_2(T_2))|\ge2$, so $j\le2k-2$. Fix a common edge $e_0$ and root $s_2(T_2)$ at $e_0$. Let $W$ be the set of $x\in V(s_1(T_1))\cap V(s_2(T_2))$ not incident to any common edge. For $x\in W$ let $f_x$ be the edge of $s_2(T_2)$ from $x$ towards $e_0$ (its parent edge). Then $f_x\notin E(s_1(T_1))$ (otherwise $x$ would be incident to a common edge), and the edges $f_x$, $x\in W$, are distinct. Let $H_1=H\setminus\{f_x:x\in W\}$. $H_1$ is connected: every vertex of $s_1(T_1)$ is joined to $e_0$ inside $s_1(T_1)\subseteq H_1$, and a vertex $y$ of $s_2(T_2)\setminus s_1(T_1)$ follows its path to $e_0$ in $s_2(T_2)$ until the first vertex in $V(s_1(T_1))$, and none of the edges used so far is deleted (deleted edges are parent edges of vertices of $W\subseteq V(s_1(T_1))$). As $s_1(T_1)$ is a tree in the connected graph $H_1$, it extends to a spanning tree $H'$ of $H_1$.

Degrees: write $\deg_i$ for the degree in $s_i(T_i)$ (zero if the vertex is not in it). If $x\in V(s_1(T_1))\cap V(s_2(T_2))$ is incident to a common edge then $\deg_{H'}(x)\le\deg_H(x)\le\deg_1(x)+\deg_2(x)-1$; if $x\in W$ then $\deg_{H'}(x)\le\deg_H(x)-1=\deg_1(x)+\deg_2(x)-1$. If $x$ lies in only one of the trees, $\deg_{H'}(x)\le\deg_H(x)=\deg_i(x)$. In all cases the exponent $\deg_{H'}(x)-1$ is at most the sum of the exponents $(\deg_i(x)-1)^+$ in $\psi(\alpha)\psi(\beta)$, and $\Delta\ge1$.
\end{proof}

\begin{lemma}[comparing $\psi(\cdot,D)$ and $\psi(\cdot,G)$]\label{lem:psi}
Let $G\in\GG^D_a$ have degree sequence $\tilde D=(\tilde d_i)$. Then for every $T\in\mathbb T_k$
\[
\sum_{s\in\mathbb S_k}\one_s(T,G)\,\Big|\frac1{\psi(s,T,D)}-\frac1{\psi(s,T,\tilde D)}\Big|\le c_k\sum_i\tilde d_i^{\,a}\le c_k\,n^{1-a}M^a ,
\]
where $c_k$ depends on $k$ and $a$.
\end{lemma}
\begin{proof}
Put $C_a=2^{1/(1-a)}$. If $d\ge1$ and $\tilde d\ge1$ are integers with $|d-\tilde d|\le d^a$, then $d\ge\tilde d/2$ and $d\le C_a\tilde d$: for $d\ge C_a$ we have $d^a\le d/2$, so $\tilde d\ge d/2$, and for $d<C_a$ use $\tilde d\ge1$; the upper bound $\tilde d\le d+d^a\le2d$ is clear. Every vertex in the image of a tree has $\tilde d\ge1$.

Fix $s$ with $\one_s(T,G)=1$ and write $e_u=b_u-1\ge0$, $a_u=d_{s(u)}^{-e_u}$, $\tilde a_u=\tilde d_{s(u)}^{-e_u}$. Then $|\prod a_u-\prod\tilde a_u|\le\sum_u|a_u-\tilde a_u|\prod_{w\ne u}\max\{a_w,\tilde a_w\}$. We have $\max\{a_w,\tilde a_w\}\le2^{e_w}\tilde d_{s(w)}^{-e_w}$, and for $e_u\ge1$, by the mean value theorem, $|a_u-\tilde a_u|\le e_u|d-\tilde d|\min\{d,\tilde d\}^{-e_u-1}\le e_u\,C_a\tilde d^{\,a}\,(\tilde d/2)^{-e_u-1}$ with $\tilde d=\tilde d_{s(u)}$, using $|d-\tilde d|\le d^a\le C_a\tilde d^a$ and $\min\{d,\tilde d\}\ge\tilde d/2$. For $e_u=0$ the term vanishes. Hence
\[
\Big|\frac1{\psi(s,T,D)}-\frac1{\psi(s,T,\tilde D)}\Big|\le c_k\sum_{u:\,b_u\ge2}\frac{\tilde d_{s(u)}^{\,a-1}}{\psi(s,T,\tilde D)} .
\]
Sum over $s$ with $\one_s(T,G)=1$ and apply Lemma~\ref{lem:peel} with $w=\one_G$, $\Delta=\tilde D$, $\tau=0$ and $\theta=a-1$: each term is at most $\sum_x\tilde d_x^{\,a}$. Finally $\sum_i\tilde d_i^a\le2^a\sum_id_i^a\le2n^{1-a}M^a$ by Hölder's inequality.
\end{proof}

\begin{proposition}[normalization]\label{prop:norm}
Let $D$ be strict graphic, $k\ge2$, $M\ge n$. Then
\begin{enumerate}[label=(\arabic*),leftmargin=2em]
\item $\displaystyle k^{k-2}\Big(1-\frac{k^2n}{2M}\Big)\le\frac1M\sum_{\mathbb S_k\times\mathbb T_k}\frac{\tp(s,T)}{\psi(s,T,D)}\le k^{k-2}$;
\item for every $G\in\GG^D$: $\displaystyle\Big|\frac1M\sum_{\mathbb S_k\times\mathbb T_k}\frac{\one_s(T,G)}{\psi(s,T,D)}-k^{k-2}\Big|\le k^{k-2}\frac{k^2n}{2M}$;
\item for every $G\in\GG^D_a$: $\displaystyle\Big|\frac1M\sum_{\mathbb S_k\times\mathbb T_k}\frac{\one_s(T,G)}{\psi(s,T,D)}-k^{k-2}\Big|\le c_k\Big(\frac nM\Big)^{1-a}$.
\end{enumerate}
\end{proposition}
\begin{proof}
(1) Lemma~\ref{lem:peel} with $w=\tp$, $\Delta=D$, $\tau=0$ (by \eqref{eq:expdeg}), $M_\Delta=M$, summed over the $k^{k-2}$ trees. (2) Same with $w=\one_G$, $\Delta=D$. (3) Let $\tilde D$ be the degree sequence of $G$ and $\tilde M=\sum\tilde d_i$, so $|\tilde M-M|\le\sum d_i^a\le n^{1-a}M^a$. By Lemma~\ref{lem:peel} for $(G,\tilde D)$ the sum with $\psi(\cdot,\tilde D)$ lies in $[\tilde M-\tfrac12nk^2,\tilde M]$ for each $T$; replacing $\psi(\cdot,\tilde D)$ by $\psi(\cdot,D)$ costs at most $c_kn^{1-a}M^a$ by Lemma~\ref{lem:psi}. Divide by $M$ and note $n/M\le(n/M)^{1-a}$.
\end{proof}

\subsection{Concentration and change of measure}\label{subsec:pf-conc}

\begin{theorem}[lower tail; proved in Appendix~\ref{app:conc}]\label{thm:janson}
Let $(J_e)_{e\in Q}$ be independent indicator variables, $Q(\alpha)\subseteq Q$ for $\alpha$ in a finite index set $\mathcal A$, $\one_\alpha=\prod_{e\in Q(\alpha)}J_e$, $\omega_\alpha>0$, $S=\sum_\alpha\omega_\alpha^{-1}\one_\alpha$. Write $\alpha\sim\beta$ if $\alpha\ne\beta$ and $Q(\alpha)\cap Q(\beta)\ne\emptyset$, and put $p_\alpha=\E\one_\alpha$, $\lambda=\E S=\sum_\alpha p_\alpha/\omega_\alpha$,
\[
\delta_1=\frac1\lambda\sum_\alpha\frac{p_\alpha}{\omega_\alpha^2},\qquad\delta_2=\frac1\lambda\sum_\alpha\sum_{\beta\sim\alpha}\frac{\E[\one_\alpha\one_\beta]}{\omega_\alpha\omega_\beta}.
\]
Then for $0\le\varepsilon\le1$
\[
\Pb\big(S\le(1-\varepsilon)\lambda\big)\le\exp\Big(-\frac\lambda{\delta_1+\delta_2}\big(\varepsilon+(1-\varepsilon)\log(1-\varepsilon)\big)\Big)\le\exp\Big(-\frac{\lambda\varepsilon^2}{2(\delta_1+\delta_2)}\Big).
\]
\end{theorem}

\noindent \emph{Reference models.} Below, $w$ denotes one of: (i) $w=\tp$ with $\Delta=D$ (and $\tau=0$); (ii) $w=q$ with $\Delta=D$ (see Section~\ref{subsec:pf-sparse}, $\tau=d_n(d_n+1)/M$); (iii) $w=\tp$ for the bipartite problem, with $\Delta=D$ and the family $\mathcal T_b$ in place of $\mathbb S_k\times\mathbb T_k$ (Section~\ref{subsec:pf-bip}). In each case $w$ satisfies \eqref{eq:wdeg} with $\Delta_i=d_i\ge1$, and $X_w$ denotes the random graph with independent edges $\Pb(ij\in X_w)=w_{ij}$, $p_w(s,T)=w(s,T)$. Let $\mathcal F$ be the family of ordered trees in question ($\mathbb S_k\times\mathbb T_k$, or $\mathcal T_b$).

\begin{proposition}[concentration]\label{prop:conc}
For $\mathcal A\subseteq\mathcal F$ put $S_{\mathcal A}(G)=\frac1M\sum_{(s,T)\in\mathcal A}\one_s(T,G)/\psi(s,T,D)$ and $\mu=\E S_{\mathcal A}(X_w)$. Then for $0\le\varepsilon\le1$
\[
\Pb\big(S_{\mathcal A}(X_w)\le(1-\varepsilon)\mu\big)\le\exp(-c_k\,\mu M\varepsilon^2) .
\]
\end{proposition}
\begin{proof}
Apply Theorem~\ref{thm:janson} with $Q=\{\text{pairs }ij\}$, $J_e=\one_{\{e\in X_w\}}$, $\alpha=(s,T)\in\mathcal A$, $Q(\alpha)=E(s(T))$, $\omega_\alpha=M\psi(s,T,D)$. Then $p_\alpha=w(s,T)$ and $\lambda=\mu$.

\emph{$\delta_1$.} Since $\psi\ge1$ and $\omega_\alpha=M\psi(\alpha)\ge M$, $\delta_1=\frac1\lambda\sum_\alpha\frac{p_\alpha}{M^2\psi(\alpha)^2}\le\frac1{M}\cdot\frac1{\lambda M}\sum_\alpha\frac{p_\alpha}{\psi(\alpha)}=\frac1M$.

\emph{$\delta_2$.} Fix $\alpha=(s_1,T_1)$. For $\beta=(s_2,T_2)\sim\alpha$ the images share an edge, and by independence $\E\one_\alpha\one_\beta=w(H)\le w(H')$ where $H=s_1(T_1)\cup s_2(T_2)$, $w(H)=\prod_{e\in E(H)}w_e$ and $H'$ is the tree of Lemma~\ref{lem:wedge} (as $w\le1$ and $E(H')\subseteq E(H)$). By Lemma~\ref{lem:wedge}, $1/(\psi(\alpha)\psi(\beta))\le1/\psi(H',D)$. Given $\alpha$ and $H'$, the number of $\beta\in\mathcal F$ with $V(s_2(T_2))\subseteq V(H')$ is at most $\binom{2k-2}k\,k!\,k^{k-2}=:c'_k$ (choose the vertex set, the tree on it, and the labelling). Hence
\[
\sum_{\beta\sim\alpha}\frac{\E\one_\alpha\one_\beta}{\psi(\alpha)\psi(\beta)}\le c'_k\sum_{H'}\frac{w(H')}{\psi(H',D)},
\]
the sum over trees $H'\supseteq s_1(T_1)$ on $k+m$ vertices, $0\le m\le k-2$. Each such $H'$ arises by adding its $m$ new vertices one at a time as leaves, in breadth-first order starting from $s_1(T_1)$: this determines a sequence of parents (at most $(2k)^{k}$ possibilities) and an injective choice of the new labels. Adding a leaf $v$ to a parent $u$ multiplies $w(\cdot)$ by $w_{s(u)s(v)}$ and $\psi$ by $\Delta_{s(u)}$, and $\sum_xw_{s(u)x}\le\Delta_{s(u)}$; so summing over the label of each new leaf does not increase the sum. Thus $\sum_{H'}w(H')/\psi(H',D)\le(2k)^kw(\alpha)/\psi(\alpha)$, and
\[
\delta_2\le\frac1{\lambda M^2}\sum_\alpha c'_k(2k)^k\frac{w(\alpha)}{\psi(\alpha)}=\frac{c'_k(2k)^k}{M},
\]
because $\sum_\alpha w(\alpha)/\psi(\alpha)=\lambda M$. (When $\mathcal F=\mathcal T_b$ all trees $H'$ above are compatible with the bipartition, so the same count applies.) So $\delta_1+\delta_2\le c_k/M$ and Theorem~\ref{thm:janson} gives the claim.
\end{proof}

\begin{proposition}[change of measure]\label{prop:transfer}
Let $w$ be as above, $\mathcal F$ the corresponding family, and let $Y$ be a random graph on $[n]$ such that $\Pb(Y\in E)\le e^{K}\Pb(X_w\in E)$ for every set $E$ of graphs ($K\ge0$). Then
\[
\frac1M\sum_{(s,T)\in\mathcal F}\frac{|p_Y(s,T)-w(s,T)|}{\psi(s,T,D)}\le c_k\sqrt{\frac{K+\log M}M}+\vartheta ,\qquad \vartheta:=\Big|\frac1M\sum_{\mathcal F}\frac{p_Y(s,T)-w(s,T)}{\psi(s,T,D)}\Big| .
\]
\end{proposition}
\begin{proof}
Let $\mathcal A^-=\{(s,T):\ w(s,T)>p_Y(s,T)\}$, $N_\pm=\frac1M\sum(p_Y-w)^\pm/\psi$. Then the left side is $N_++N_-=2N_-+(N_+-N_-)\le2N_-+\vartheta$, so we bound $N_-$. Let $S=S_{\mathcal A^-}$, $\mu=\E S(X_w)$; then $\E S(Y)=\mu-N_-$. By Lemma~\ref{lem:peel}, $\mu\le\frac1M\sum_{\mathcal F}w/\psi\le k^{k-2}$. Put $\varepsilon^2=2(K+\log M)/(c_k'\mu M)$ with $c_k'$ the constant of Proposition~\ref{prop:conc}.

If $\varepsilon\le1$: by hypothesis and Proposition~\ref{prop:conc}, $\Pb(S(Y)\le(1-\varepsilon)\mu)\le e^K\exp(-c'_k\mu M\varepsilon^2)\le M^{-2}$. As $S\ge0$, $\E S(Y)\ge(1-\varepsilon)\mu(1-M^{-2})$, so $N_-\le\varepsilon\mu+\mu M^{-2}\le c_k\sqrt{(K+\log M)/M}$ (using $\mu\le k^{k-2}$).
If $\varepsilon>1$ then $\mu<2(K+\log M)/(c_k'M)$, and $N_-\le\mu\le\sqrt{k^{k-2}\mu}\le c_k\sqrt{(K+\log M)/M}$.
\end{proof}

\subsection{\texorpdfstring{A lower bound for $\Pb(\widetilde\cG(D)\in\GG^D_a)$}{A lower bound for the probability that the independent model has almost given degrees}}\label{subsec:pf-lower}

For $i\in[n]$ and $S\subseteq[n]$ write $\rho(i,S)=\sum_{j\in S\setminus\{i\}}\tp_{ij}$, and for $A,B\subseteq[n]$ let $\Ed(A,B)$ be the set of pairs $\{i,j\}$ with $i\in A$, $j\in B$. For $E\subseteq\Ed([n],[n])$ let $\widetilde\cG_E$ be the restriction of $\widetilde\cG(D)$ to $E$, and $\widetilde\cG_A=\widetilde\cG_{\Ed(A,A)}$, $\widetilde\cG_{A,B}=\widetilde\cG_{\Ed(A,B)}$. For $\beta>0$ let $I_\beta=\{i:\ d_i\le\log^\beta n\}$ (an initial segment of $[n]$).

\begin{lemma}[rounding]\label{lem:round}
For every partition $[n]=A\sqcup B$ there is a bipartite graph $T\subseteq\Ed(A,B)$ with
\[
\deg_T(i)\in\{\lfloor\rho(i,B)\rfloor,\lfloor\rho(i,B)\rfloor+1\}\ (i\in A),\qquad
\deg_T(j)\in\{\lfloor\rho(j,A)\rfloor,\lfloor\rho(j,A)\rfloor+1\}\ (j\in B).
\]
\end{lemma}
\begin{proof}
Give each edge $e\in\Ed(A,B)$ the weight $w(e)=\tp_e\in[0,1]$; the weighted degree of $i\in A$ is $\rho(i,B)$ and of $j\in B$ is $\rho(j,A)$. Call an edge \emph{live} if $0<w(e)<1$ and \emph{dead} otherwise. \emph{Stage 1.} While the live edges contain a cycle $e_1,\dots,e_{2m}$ (even, since the graph is bipartite), add $+c$ to the weights of $e_1,e_3,\dots$ and $-c$ to those of $e_2,e_4,\dots$, increasing $c$ from $0$ until some edge dies. All weighted degrees are unchanged and the number of live edges decreases. \emph{Stage 2.} Now the live edges form a forest. While live edges remain, take a longest path $v_1,\dots,v_m$ of live edges; $v_1$ and $v_m$ are incident to exactly one live edge each (otherwise the path could be extended or a cycle would exist). Add $+c,-c,+c,\dots$ alternately along the path, increasing $c$ until an edge dies. Only the weighted degrees of $v_1,v_m$ change.

Consider a vertex $v$ with initial weighted degree $\rho$. Its weighted degree changes only at times when it is an endpoint of the chosen path, i.e.\ when it has exactly one live edge $e$ and some number $q$ of edges of weight $1$; then the weighted degree is $q+w(e)$ with $0<w(e)<1$, so $q=\lfloor\rho\rfloor$ (it equals $\rho$ up to this time) and $\rho\notin\Z$. While $e$ is live the degree stays in $(q,q+1)$, and when $e$ dies it becomes $q$ or $q+1$. If $v$ is never such an endpoint, its weighted degree is unchanged, hence is an integer at the end (all edges are dead). In all cases the final degree lies in $\{\lfloor\rho\rfloor,\lfloor\rho\rfloor+1\}$; the final weights are $0$ or $1$ and define $T$.
\end{proof}

\begin{lemma}\label{lem:small}
Let $M\ge n^{1+\varepsilon}$, $\beta>0$, $A=I_\beta\ne\emptyset$. Then $\rho(i,A)\to0$ as $n\to\infty$, uniformly in $i\in A$ and in $D$; in particular $\rho(i,A)<\frac14$ for large $n$.
\end{lemma}
\begin{proof}
Let $k=|A|$, $\Lambda=\log^\beta n$, so $d_k\le\Lambda$, and put $\ell=n-2d_k$. The vertices $\{i\le k\}\cup\{i>\ell\}$ have total degree $\le kd_k+2d_kn\le3\Lambda n<M$, hence they do not exhaust $[n]$, i.e.\ $k<\ell$. Since $\tp_{kj}$ is increasing in $j$, $d_k\ge\sum_{j>\ell}\tp_{kj}\ge2d_k\tp_{k\ell}$, so $\tp_{k\ell}\le\frac12$, i.e.\ $r_kr_\ell\le1$, and also $r_k\le1$. Hence $r_ir_j\le r_kr_\ell\le1$ for $i\le k$, $j\le\ell$, and so $\tfrac12r_ir_j\le\tp_{ij}\le r_ir_j$ for such pairs.

\emph{Upper bound for $\sum_{i\le k}r_i$.} $(\sum_{i\le k}r_i)^2\le2\sum_{i<j\le k}r_ir_j+\sum_{i\le k}r_i^2\le4\sum_{i<j\le k}\tp_{ij}+k\le2\sum_{i\le k}d_i+n\le(2\Lambda+1)n$.

\emph{Lower bound for $\sum_{i\le\ell}r_i$.} $(\sum_{i\le\ell}r_i)^2\ge\sum_{i\ne j\le\ell}r_ir_j\ge\sum_{i\ne j\le\ell}\tp_{ij}\ge\sum_{i\le\ell}d_i-\sum_{j>\ell}d_j=M-2\sum_{j>\ell}d_j\ge M-4d_kn\ge M/2$.

For $j\le k$: $d_j\ge\sum_{i\le\ell,i\ne j}\tp_{ij}\ge\frac{r_j}2\big(\sum_{i\le\ell}r_i-r_j\big)$, and for large $n$, $\sum_{i\le\ell}r_i\ge\sqrt{M/2}\ge2\sqrt{(2\Lambda+1)n}\ge2r_j$, so $r_j\le4d_j/\sum_{i\le\ell}r_i$. Therefore, for $j\le k$,
\[
\rho(j,A)\le r_j\sum_{i\le k}r_i\le\frac{4d_j\sqrt{(2\Lambda+1)n}}{\sqrt{M/2}}\le4\Lambda\sqrt{\frac{2(2\Lambda+1)n}{M}}\le C\Lambda^{3/2}n^{-\varepsilon/2}\to0 .\qedhere
\]
\end{proof}

\begin{proposition}\label{prop:lower}
Let $D$ be strict graphic with $M\ge n^{1+\varepsilon}$ ($M$ even), let $\frac12<a<1$, $a_1=a-\frac12$, $\delta=\min_{i\ne j}\min\{\tp_{ij},1-\tp_{ij}\}$ and
\begin{equation}\label{eq:C1}
C_1(D)=|\log\delta|\,n\,\log^{10/a_1}n .
\end{equation}
Then for large $n$
\[
\Pb\big(\widetilde\cG(D)\in\GG^D_a\big)\ge e^{-11\,C_1(D)} .
\]
\end{proposition}
\begin{proof}
Put $\beta=10/a_1$, $A=I_\beta$, $B=[n]\setminus A$, $\Lambda=\log^\beta n$, and let $T$ be the graph of Lemma~\ref{lem:round}. The three families of edges $\Ed(A,A)$, $\Ed(A,B)$, $\Ed(B,B)$ are disjoint, so $\widetilde\cG_A,\widetilde\cG_{A,B},\widetilde\cG_B$ are independent. Note $|\log\delta|\ge\log2$, $\sum_{i\in A}d_i\le n\Lambda$, and $\sum_{j\ne i}\tp_{ij}=d_i$.

\emph{(a) $\Pb(\widetilde\cG_A=\emptyset)\ge e^{-3|\log\delta|n\Lambda}$.} Let $F_1$ be the set of pairs $\{i,j\}\subseteq A$ with $\tp_{ij}>\frac12$. Since $\sum_{\{i,j\}\subseteq A}\tp_{ij}\le\frac12\sum_{i\in A}d_i\le\frac12n\Lambda$, $|F_1|<n\Lambda$. For $\tp\le\frac12$, $1-\tp\ge e^{-2\tp}$, and for pairs in $F_1$, $1-\tp\ge\delta$. Hence $\Pb(\widetilde\cG_A=\emptyset)\ge\delta^{|F_1|}e^{-2\cdot\frac12n\Lambda}\ge e^{-(|\log\delta|+1)n\Lambda}$, and $1\le2|\log\delta|$.

\emph{(b) $\Pb(\widetilde\cG_{A,B}=T)\ge e^{-7|\log\delta|n\Lambda}$.} By Lemma~\ref{lem:round}, $|E(T)|\le\sum_{i\in A}(\rho(i,B)+1)\le n\Lambda+n\le2n\Lambda$. Let $F_2$ be the set of pairs in $\Ed(A,B)$ with $\tp>\frac12$; as $\sum_{\Ed(A,B)}\tp\le\sum_{i\in A}d_i\le n\Lambda$, $|F_2|<2n\Lambda$. Then
\[
\Pb(\widetilde\cG_{A,B}=T)\ge\delta^{|E(T)|+|F_2|}\,e^{-2n\Lambda}\ge e^{-4|\log\delta|n\Lambda-2n\Lambda}\ge e^{-7|\log\delta|n\Lambda},
\]
because $2\le3|\log\delta|$. Here the factor $\delta$ is used for the edges of $T$ (where $\tp\ge\delta$) and for the pairs of $F_2$ outside $T$ (where $1-\tp\ge\delta$), and $1-\tp\ge e^{-2\tp}$ for the remaining pairs, whose $\tp$ sum to at most $n\Lambda$.

\emph{(c)} Let $L=\log^{1/a_1}n$ and $J=\{j\in B:\ \rho(j,B)\ge L\}$. Note $1/a_1>2$ because $a<1$. Let $\mathcal E$ be the event that
\begin{itemize}[leftmargin=2em]
\item for $j\in J$: $|\deg_{\widetilde\cG_B}(j)-\rho(j,B)|\le\frac12\rho(j,B)^a$;
\item for $j\in B\setminus J$: $\deg_{\widetilde\cG_B}(j)\le4L$.
\end{itemize}
$\deg_{\widetilde\cG_B}(j)$ is a sum of independent indicators with mean $\rho(j,B)$. By Chernoff's bound $\Pb(|X-\mu|\ge\eta\mu)\le2e^{-\eta^2\mu/3}$ ($0<\eta\le1$) with $\eta=\frac12\mu^{a-1}$, the first requirement fails with probability $\le2\exp(-\mu^{2a_1}/12)\le2\exp(-\log^2n/12)$. By $\Pb(X\ge x)\le(e\mu/x)^x$ ($x>\mu$) with $x=4L>\mu$, the second fails with probability $\le(e/4)^{4L}\le e^{-\log^2n}$. By the union bound, $\Pb(\mathcal E)\ge1-n^{-1}\ge\frac12$.

\emph{On $\{\widetilde\cG_A=\emptyset\}\cap\{\widetilde\cG_{A,B}=T\}\cap\mathcal E$ we have $\widetilde\cG\in\GG^D_a$.} The graph $\widetilde\cG$ is then $T\cup\widetilde\cG_B$. Let $s_i=|d_i-\deg_{T\cup\widetilde\cG_B}(i)|$, an integer; we must show $s_i\le d_i^a$. Recall $d_i=\rho(i,A)+\rho(i,B)$.
\begin{itemize}[leftmargin=2em]
\item $i\in A$: $\deg(i)=\deg_T(i)$, so $s_i\le\rho(i,A)+|\rho(i,B)-\deg_T(i)|<\rho(i,A)+1<2$ by Lemmas~\ref{lem:round} and~\ref{lem:small}; thus $s_i\le1\le d_i^a$.
\item $i\in J$: $s_i\le|\rho(i,A)-\deg_T(i)|+|\rho(i,B)-\deg_{\widetilde\cG_B}(i)|<1+\frac12\rho(i,B)^a\le\rho(i,B)^a\le d_i^a$, since $\rho(i,B)^a\ge L^a\ge2$.
\item $i\in B\setminus J$: $\rho(i,B)<L$ and $\deg_{\widetilde\cG_B}(i)\le4L$, so $s_i<1+4L$; since $i\notin I_\beta$, $d_i^a>\log^{a\beta}n=\log^{10(1+1/(2a_1))}n\ge 1+4\log^{1/a_1}n$.
\end{itemize}
Multiplying (a), (b) and the probability $\ge\frac12$ of (c): $\Pb(\widetilde\cG\in\GG^D_a)\ge\frac12e^{-10|\log\delta|n\Lambda}=\frac12e^{-10C_1}\ge e^{-11C_1}$.
\end{proof}

\subsection{Proof of Theorem~\ref{thm:almost} and Remark~\ref{rem:best}}\label{subsec:pf-almost}

\begin{lemma}\label{lem:transfer-a}
Let $D$ be strict graphic with $M\ge n^{1+\varepsilon}$ and $C_1$ as in \eqref{eq:C1}, $C_2(D)=\sum_id_i^a|\log r_i|$. For every set $E$ of graphs
\[
\Pb(\cG_a(a,D)\in E)\le e^{11C_1(D)+2C_2(D)}\,\Pb(\widetilde\cG(D)\in E).
\]
\end{lemma}
\begin{proof}
Let $P_{\min},P_{\max}$ be the minimum and maximum of $\Pb(\widetilde\cG=G)$ over $G\in\GG^D_a$. By \eqref{eq:probG}, for $G_1,G_2\in\GG^D_a$, $\log\frac{\Pb(\widetilde\cG=G_1)}{\Pb(\widetilde\cG=G_2)}=\sum_i(d_i(G_1)-d_i(G_2))\log r_i\le2C_2$, as $|d_i(G_1)-d_i(G_2)|\le2d_i^a$. Then $\Pb(\cG_a\in E)=\frac{|E\cap\GG^D_a|}{|\GG^D_a|}\le\frac{\Pb(\widetilde\cG\in E)/P_{\min}}{\Pb(\widetilde\cG\in\GG^D_a)/P_{\max}}$ since $|E\cap\GG^D_a|\le\Pb(\widetilde\cG\in E\cap\GG^D_a)/P_{\min}$ and $\Pb(\widetilde\cG\in\GG^D_a)\le P_{\max}|\GG^D_a|$. Apply Proposition~\ref{prop:lower}.
\end{proof}

\begin{proof}[Proof of Theorem~\ref{thm:almost}]
Put $\Theta=d_n-d_\ell+1$; Assumption~\ref{ass:type} gives $\Theta<n^{-\nu}\sqrt{M/n}$, and in particular $\sqrt{n/M}<n^{-\nu}$; as $M<n^2$, this forces $\nu<\frac12$. By Lemma~\ref{lem:logr},
\[
C_1\le8\Theta\,n\log^{10/a_1+1}n\le8n^{-\nu}\sqrt{nM}\log^{10/a_1+1}n\le M\,n^{-\nu},
\]
for large $n$ (since $\sqrt{n/M}\le n^{-\varepsilon/2}$), and, by Hölder's inequality $\sum d_i^a\le n^{1-a}M^a$ and $M\le n^2$,
\[
C_2\le4\Theta\log n\sum_id_i^a\le4n^{-\nu}\sqrt{\tfrac Mn}\,n^{1-a}M^a\log n=4n^{-\nu}M\Big(\frac Mn\Big)^{a_1}\log n\le4\,n^{a_1-\nu}M\log n .
\]
Lemma~\ref{lem:transfer-a} shows that $Y=\cG_a$ satisfies the hypothesis of Proposition~\ref{prop:transfer} (with $w=\tp$) with $K=11C_1+2C_2\le c\,n^{a_1-\nu}M\log n$. By Proposition~\ref{prop:norm}(1),(3), $\vartheta\le c_k(n/M)^{1-a}$. Hence
\[
L_{\mathbf a}(a,k,D)\le c_k\Big(\sqrt{n^{a_1-\nu}\log n}+\Big(\frac nM\Big)^{1-a}\Big).
\]
Since $\log n\le n^{a_1}$, the first term is $\le n^{(a_1-\nu)/2}n^{a_1/2}=n^{a_1-\nu/2}$. For the second, $(n/M)^{1-a}=\big(\sqrt{n/M}\big)^{2(1-a)}<n^{-2\nu(1-a)}=n^{-\nu+2\nu a_1}\le n^{a_1-\nu/2}$, because $a_1(2\nu-1)\le\nu/2$ (recall $\nu<\frac12$). Hence $L_{\mathbf a}\le c_k\,n^{a_1-\nu/2}=c_k\,n^{a-\frac{\nu+1}2}$, which is $o(1)$ for $a<\frac\nu4+\frac12$.
\end{proof}

\begin{remark}
The proof gives the stronger bound $L_{\mathbf a}\le c_k\big(n^{(a_1-\nu)/2}\sqrt{\log n}+(n/M)^{1-a}\big)$, which is $o(1)$ already for $a<\nu+\frac12$. We keep the range of the original statement.
\end{remark}

\begin{proof}[Proof of Remark~\ref{rem:best}]
If $|\log r_i|\le c\log n$ then $|\log\delta|\le c\log n$, so $C_1\le c\,n\log^{10/a_1+1}n$ and $C_2\le c\log n\sum_id_i^a\le c\log n\,M(n/M)^{1/2}n^{a_1}$ (Hölder's inequality and $M\le n^2$: $n^{1-a}M^a=M(n/M)^{1-a}=M(n/M)^{1/2}(M/n)^{a_1}$). With Lemma~\ref{lem:transfer-a} and Proposition~\ref{prop:transfer}, $\sqrt{K/M}\le c\,(n/M)^{1/4}n^{a_1/2}\log^{c'}n\le c\,(n/M)^{1/4}n^{a_1}$, and $\vartheta\le c_k(n/M)^{1/2-a_1}\le c_k(n/M)^{1/4}n^{a_1}$ because $a_1\le\frac14$.
\end{proof}

\subsection{Graphs with a given degree sequence}\label{subsec:pf-exact}

\begin{proof}[Proof of Proposition~\ref{prop:conj}]
By Lemma~\ref{lem:prob}, $\widetilde\cG(D)$ conditioned on $\GG^D$ is uniform, so for every set $E$ of graphs $\Pb(\cG_g\in E)=\Pb(\widetilde\cG\in E\cap\GG^D)/\pi(D)\le e^{K_0}\Pb(\widetilde\cG\in E)$. Apply Proposition~\ref{prop:transfer} with $K=K_0$ and note that by Proposition~\ref{prop:norm}(1),(2), $\vartheta\le k^{k-2}k^2n/M$. Finally, if $K_0=\eta n\log n$ then $\sqrt{(K_0+\log M)/M}\le c\sqrt{n\log n/M}$ and $n/M\le\sqrt{n\log n/M}$ (as $M\ge n$), which gives the last assertion.
\end{proof}

\subsection{Proof of Theorem~\ref{thm:dense}}\label{subsec:pf-dense}
Since $D$ is $\delta$-tame, $\delta\le\tp_{ij}\le1-\delta$; thus $r_ir_j\in[\delta',1/\delta']$ for all $i\ne j$ with $\delta'=\delta/(1-\delta)$, and, using $r_i^2=(r_ir_j)(r_ir_l)/(r_jr_l)$, $|\log r_i|\le c(\delta)$ for all $i$. Also $M\ge c_2n(n-1)$.

\emph{Exact degrees.} By the enumeration theorem for $\delta$-tame degree sequences of \cite{BH13} (see also \cite{CDS11}), $\pi(D)=|\GG^D|e^{-H_1(\tp)}$ equals $(1+o(1))\,2\,(2\pi)^{-n/2}(\det Q)^{-1/2}e^{O(1)}$, where $Q$ is the covariance matrix of the degrees of $\widetilde\cG(D)$ (entries $Q_{ii}=\sum_{j}\tp_{ij}(1-\tp_{ij})\le n$, off-diagonal entries in $[0,\tfrac14]$) and the $O(1)$ is bounded by a constant depending on $\delta$. By Hadamard's inequality $\det Q\le\prod_iQ_{ii}\le n^n$. Therefore $\pi(D)\ge e^{-\eta n\log n}$, which is \eqref{eq:vol}, and Proposition~\ref{prop:conj} gives \eqref{eq:densecase}, since $M\asymp n^2$ and $\sqrt{n\log n/M}\asymp\sqrt{\log n/n}$.

\emph{Almost given degrees.} Here $C_2\le c\sum_id_i^a\le c\,n^{1+a}$, and $C_1\le c\,n\log^{10/a_1}n$. By Lemma~\ref{lem:transfer-a}, Proposition~\ref{prop:transfer} and Proposition~\ref{prop:norm}, $L_{\mathbf a}\le c_k\big(\sqrt{n^{1+a}/n^2}+(n/M)^{1-a}\big)\le c_kn^{(a-1)/2}\le c_kn^{-\frac34+a}$, the last step because $a\ge\frac12$. \qed

\subsection{Very sparse graphs}\label{subsec:pf-sparse}

Let $q_{ij}=d_id_j/(M+d_id_j)$, $r_i^{(q)}=d_i/\sqrt M$, so $q_{ij}=r^{(q)}_ir^{(q)}_j/(1+r^{(q)}_ir^{(q)}_j)$, and $\tau=d_n(d_n+1)/M$.

\begin{lemma}\label{lem:q}
Let $D$ be a graphic sequence with $d_i\ge1$ and $\tau<1$, and let $d_{q,i}=\sum_{j\ne i}q_{ij}$. Then:
\begin{enumerate}[label=(\arabic*),leftmargin=2em]
\item $(1-\tau)d_i\le d_{q,i}\le d_i$; hence \eqref{eq:wdeg} holds for $w=q$, $\Delta=D$;
\item $\Pb(\cG_q=G)=\prod_i(r_i^{(q)})^{d_i(G)}/\prod_{i<j}(1+r_i^{(q)}r_j^{(q)})$ for every graph $G$; hence $\cG_q$ conditioned on $\GG^D$ is uniform, and for $G_1,G_2\in\GG^D_a$ one has $\Pb(\cG_q=G_1)\le e^{2\log n\sum_id_i^a}\Pb(\cG_q=G_2)$;
\item if $d_n=o(M^{1/4})$ then $\Pb(\cG_q\in\GG^D)\ge e^{-c\,n\log n}$.
\end{enumerate}
\end{lemma}
\begin{proof}
(1) Since $\sum_{j\ne i}d_id_j/M=d_i-d_i^2/M$,
\[
d_i-d_{q,i}=\frac{d_i^2}M+\sum_{j\ne i}\frac{d_i^2d_j^2}{M(M+d_id_j)}\in\Big[0,\ \frac{d_i^2}M+\frac{d_i^2}{M^2}\sum_jd_j^2\Big]\subseteq\Big[0,\frac{d_i^2(1+d_n)}M\Big]\subseteq[0,\tau d_i],
\]
using $\sum_jd_j^2\le d_nM$.

(2) Same computation as in Lemma~\ref{lem:prob}. Since $1/n\le d_i/\sqrt M\le n$ (as $1\le d_i<n$ and $M\le n^2$), $|\log r^{(q)}_i|\le\log n$, and the ratio bound follows as in the proof of Lemma~\ref{lem:transfer-a}.

(3) We use McKay's formula \cite[Thm.~1.1]{McK10} (see \cite{McK85}): if $d_n=o(M^{1/4})$ and $M$ is even,
\[
|\GG^D|=\frac{M!}{(M/2)!\,2^{M/2}\prod_id_i!}\exp\Big(-\mu-\mu^2+O(d_n^4/M)\Big),\qquad\mu=\frac{\sum_id_i(d_i-1)}{2M}.
\]
By (2), $\log\Pb(\cG_q\in\GG^D)=\log|\GG^D|+\sum_id_i\log\frac{d_i}{\sqrt M}-\sum_{i<j}\log(1+x_{ij})$, $x_{ij}=d_id_j/M\le d_n^2/M<1$. Using Stirling's formula for $M!$, $(M/2)!$ and $d_i!$, the terms $\pm\frac M2\log M$ and $\pm\sum_id_i\log d_i$ cancel and
\[
\log|\GG^D|+\sum_id_i\log\tfrac{d_i}{\sqrt M}=\tfrac M2+\tfrac12\log2-\tfrac12\sum_i\log(2\pi d_i)-O(n)-\mu-\mu^2+o(1).
\]
Next, $\log(1+x)\le x-\frac{x^2}2+\frac{x^3}3$, $\sum_{i<j}x_{ij}=\frac M2-\frac{\sum d_i^2}{2M}$, $\sum_{i<j}x_{ij}^2=\frac{(\sum d_i^2)^2-\sum d_i^4}{2M^2}$ and $\sum_{i<j}x_{ij}^3\le\frac{(\sum d_i^3)^2}{2M^3}\le\frac{d_n^4}{2M}$. Since $\mu=\frac{\sum d_i^2}{2M}-\frac12$, a direct computation gives
\[
-\sum_{i<j}\log(1+x_{ij})-\mu-\mu^2\ \ge\ -\frac M2+\frac14+\frac{\sum d_i^2}{2M}-\frac{\sum d_i^4}{4M^2}-\frac{d_n^4}{6M}\ \ge\ -\frac M2-\frac{d_n^3}{4M}-\frac{d_n^4}{6M}.
\]
Altogether $\log\Pb(\cG_q\in\GG^D)\ge-\frac12\sum_i\log(2\pi d_i)-O(n)-o(1)\ge-c\,n\log n$.
\end{proof}

\begin{lemma}[normalization for $q$]\label{lem:normq}
For $\tau<1$: $\ k^{k-2}\big((1-\tau)^{k-1}-\tfrac{k^2n}{2M}\big)\le\frac1M\sum_{\mathbb S_k\times\mathbb T_k}\frac{p_q(s,T)}{\psi(s,T,D)}\le k^{k-2}$.
\end{lemma}
\begin{proof}
Lemma~\ref{lem:peel} with $w=q$, $\Delta=D$ (allowed by Lemma~\ref{lem:q}(1)).
\end{proof}

\begin{proof}[Proof of Theorem~\ref{thm:sparse}]
\emph{(1)} For large $n$, $d_n(d_n+1)<M/2$, so $\ell=n$, $\Theta=1$, and Remark~\ref{rem:r} gives $|\log r_i|\le3\log n$; the hypothesis $(n/M)^{1/2}n^{3a_1}\ll1$ implies $n^{3a_1}\ll\sqrt{M/n}\le\sqrt n$, so $a_1<\frac16$. Remark~\ref{rem:best} gives $L_{\mathbf a}\le c_k(n/M)^{1/4}n^{a_1}$.

\emph{(2),(3).} Here $d_n=o(M^{1/4})$ is used; recall $\tau\le2d_n^2/M$.

\emph{Step 1: $\cG_a$ versus $\cG_q$.} For a set $E$ of graphs, by Lemma~\ref{lem:q}(2) as in Lemma~\ref{lem:transfer-a},
\[
\Pb(\cG_a\in E)\le\frac{e^{2\log n\sum d_i^a}\,\Pb(\cG_q\in E)}{\Pb(\cG_q\in\GG^D_a)}\le e^{K}\Pb(\cG_q\in E),\qquad K=2\log n\sum_id_i^a+c\,n\log n,
\]
because $\Pb(\cG_q\in\GG^D_a)\ge\Pb(\cG_q\in\GG^D)\ge e^{-cn\log n}$. By Hölder, $\sum d_i^a\le M(n/M)^{1/2}n^{a_1}$, so $K/M\le2\log n\,(n/M)^{1/2}n^{a_1}+c\,n\log n/M$. By Proposition~\ref{prop:transfer} with $w=q$ and Lemma~\ref{lem:normq} together with Proposition~\ref{prop:norm}(3) (so $\vartheta\le c_k(\tau+(n/M)^{1-a})$),
\[
I:=\frac1M\sum\frac{|p_a-p_q|}{\psi(D)}\le c_k\Big(\Big(\frac nM\Big)^{1/4}n^{a_1}+\frac{d_n^2}M\Big),
\]
where we used $(n/M)^{1-a}=(n/M)^{1/2-a_1}\le(n/M)^{1/4}n^{a_1}$ (as $a_1<\frac16$).

\emph{Step 2: $\cG_g$ versus $\cG_q$.} By Lemma~\ref{lem:q}(2), $\cG_q$ conditioned on $\GG^D$ is uniform, so $\Pb(\cG_g\in E)\le e^{cn\log n}\Pb(\cG_q\in E)$. Proposition~\ref{prop:transfer}, Lemma~\ref{lem:normq} and Proposition~\ref{prop:norm}(2) give
\[
\frac1M\sum\frac{|p_g-p_q|}{\psi(D)}\le c_k\Big(\sqrt{\frac{n\log n}M}+\frac{d_n^2}M\Big).
\]

\emph{Step 3.} By the triangle inequality, $L_{\mathbf q}\le I+L_{\mathbf a}$ and $L_{\mathbf g}\le L_{\mathbf q}+\frac1M\sum|p_g-p_q|/\psi$. Together with (1) and Steps 1--2 this gives (3) and (2); note $\sqrt{n\log n/M}\le(n/M)^{1/4}n^{a_1}$.
\end{proof}
\begin{remark}
The restriction $d_n=o(M^{1/4})$ in (2),(3) comes only from the use of McKay's formula in Lemma~\ref{lem:q}(3); the original version stated them under $d_n^2=o(M)$, where the cited enumeration result does not apply. Part (1), which needs no enumeration, holds for $d_n^2=o(M)$.
\end{remark}

\subsection{Proof of Theorem~\ref{thm:bip}}\label{subsec:pf-bip}
Extend $\tp$ by $\tp_{ij}=0$ for $i,j$ in the same part; then $\sum_{x\ne i}\tp_{ix}=d_{1,i}$ (or $d_{2,i}$) and \eqref{eq:wdeg} holds with $\tau=0$, $\Delta=D$, and $p_b(s,T)=\tp(s,T)=0$ unless $s(T)$ is compatible with the bipartition, so sums over $\mathcal T_b$ equal sums over $\mathbb S_k\times\mathbb T_k$. Lemma~\ref{lem:peel}, Proposition~\ref{prop:conc} and Proposition~\ref{prop:norm}(1),(2) (the proofs use only \eqref{eq:wdeg}) apply verbatim. As in Lemma~\ref{lem:prob}, $\Pb(\widetilde\cG_b=G)=\prod_ir_{1,i}^{d_{1,i}}\prod_jr_{2,j}^{d_{2,j}}/\prod_{i,j}(1+r_{1,i}r_{2,j})$ depends only on the degrees, so $\widetilde\cG_b$ conditioned on $\GG^{D_1,D_2}$ is uniform. By \cite[Thm.~1.1]{Bar10a}, $\pi_b:=\Pb(\widetilde\cG_b\in\GG^{D_1,D_2})\ge(n_1n_2)^{-\eta(n_1+n_2)}\ge e^{-2\eta n\log n}$ whenever $\PP^{D_1,D_2}$ has an interior point. Proposition~\ref{prop:transfer} with $K=2\eta n\log n$ gives $L_{\mathbf b}\le c_k(\sqrt{n\log n/M}+n/M)\le c_k\sqrt{n\log n/M}$. \qed

\section*{Acknowledgements}
The author would like to thank S.~R.~S.~Varadhan for reviewing the manuscript and for his useful comments, S.~Chatterjee for suggesting the problem, and A.~Barvinok, M.~Harel, A.~Krishnan and A.~Munez for fruitful discussions and references.

\appendix
\section{Proof of Theorem~\ref{thm:janson}}\label{app:conc}
Write $\Delta=\delta_1+\delta_2$. For $t\ge0$ let $\Psi(t)=\E e^{-tS}$. Then $-\Psi'(t)=\sum_\alpha\omega_\alpha^{-1}\E[\one_\alpha e^{-tS}]$. For fixed $\alpha$ split $S=S'_\alpha+S''_\alpha$, with
\[
S'_\alpha=\omega_\alpha^{-1}\one_\alpha+\sum_{\beta\sim\alpha}\omega_\beta^{-1}\one_\beta,\qquad S''_\alpha=\sum_{\beta\ne\alpha,\ \beta\not\sim\alpha}\omega_\beta^{-1}\one_\beta .
\]
$S''_\alpha$ is independent of $\one_\alpha$. Conditionally on $\{\one_\alpha=1\}$ (which fixes $J_e$, $e\in Q(\alpha)$), both $e^{-tS'_\alpha}$ and $e^{-tS''_\alpha}$ are decreasing functions of the remaining independent variables, so by Harris' (FKG) inequality
\[
\E[\one_\alpha e^{-tS}]\ge p_\alpha\,\E[e^{-tS'_\alpha}\mid\one_\alpha=1]\ \E[e^{-tS''_\alpha}]\ge p_\alpha\,\E[e^{-tS'_\alpha}\mid\one_\alpha=1]\,\Psi(t),
\]
because $S\ge S''_\alpha$. Hence, using Jensen's inequality twice (the weights $p_\alpha/(\omega_\alpha\lambda)$ sum to $1$),
\[
-\frac{\Psi'(t)}{\Psi(t)}\ge\sum_\alpha\frac{p_\alpha}{\omega_\alpha}\E[e^{-tS'_\alpha}\mid\one_\alpha=1]\ge\lambda\exp\Big(-\frac t\lambda\sum_\alpha\frac{p_\alpha}{\omega_\alpha}\E[S'_\alpha\mid\one_\alpha=1]\Big)=\lambda e^{-t\Delta},
\]
since $\sum_\alpha\frac{p_\alpha}{\omega_\alpha}\E[S'_\alpha\mid\one_\alpha=1]=\sum_\alpha\omega_\alpha^{-1}\E[\one_\alpha S'_\alpha]=\sum_\alpha\frac{p_\alpha}{\omega_\alpha^2}+\sum_\alpha\sum_{\beta\sim\alpha}\frac{\E\one_\alpha\one_\beta}{\omega_\alpha\omega_\beta}=\lambda\Delta$ (here $\one_\alpha^2=\one_\alpha$). Integrating from $0$ to $t$ and using $\Psi(0)=1$: $\Psi(t)\le\exp\big(-\frac\lambda\Delta(1-e^{-\Delta t})\big)$. By Markov's inequality, for $0\le\varepsilon<1$,
\[
\Pb(S\le(1-\varepsilon)\lambda)\le e^{t(1-\varepsilon)\lambda}\Psi(t)\le\exp\Big(t(1-\varepsilon)\lambda-\frac\lambda\Delta(1-e^{-\Delta t})\Big).
\]
Choosing $e^{-\Delta t}=1-\varepsilon$, i.e.\ $t=-\Delta^{-1}\log(1-\varepsilon)$, gives the exponent $-\frac\lambda\Delta\big(\varepsilon+(1-\varepsilon)\log(1-\varepsilon)\big)$; the case $\varepsilon=1$ follows by letting $\varepsilon\uparrow1$. Finally $f(\varepsilon)=\varepsilon+(1-\varepsilon)\log(1-\varepsilon)$ has $f(0)=f'(0)=0$ and $f''(\varepsilon)=1/(1-\varepsilon)\ge1$, so $f(\varepsilon)\ge\varepsilon^2/2$. \qed


\end{document}